\documentclass[12pt]{amsart}
\usepackage{amsmath}
\usepackage{amssymb}
\usepackage{amsfonts}
\usepackage{amsthm}
\usepackage{mathrsfs}
\usepackage[all]{xy}

\newcommand{\be}{\begin{equation}}
\newcommand{\ee}{\end{equation}}
\newtheorem{theorem}{Theorem}[section]
\newtheorem{lemma}[theorem]{Lemma}

\newtheorem{proposition}[theorem]{Proposition}

\theoremstyle{definition}
\newtheorem{definition}[theorem]{Definition}

\theoremstyle{remark}
\newtheorem{remark}[theorem]{Remark}

\numberwithin{equation}{section}

\begin{document}

\title{Isomorphism between two realizations of Yangian $Y(\mathfrak{so}_3)$}
\author{Naihuan Jing, Ming Liu$^*$}
\address{NJ: Department of Mathematics, North Carolina State University, Raleigh, NC 27695, USA}
\address{ML: Chern Institute of Mathematics, Nankai University, Tianjin, 300071, China}

\thanks{{\scriptsize
\hskip -0.4 true cm MSC (2010): Primary: 17B30; Secondary: 17B68.
$*$Corresponding author.
}}

\maketitle

\begin{abstract}  The isomorphism between Drinfeld's new realization
and the FRT realization is proved for the Yangian algebra $Y(\mathfrak{so}_3)$
by using Gauss decomposition.
\end{abstract}
\section{Introduction}
When the concept of quantum groups was introduced in 1985 two main classes of examples are
Drinfeld-Jimbo quantum enveloping algebras of symmetrizable Kac-Moody Lie algebras
\cite{D1, J}
and the Yangian algebra \cite{D1} associated with the complex simple Lie algebra.
These quantum algebras correspond respectively to the trigonometric and rational solutions of the quantum Yang-Baxter equation.
They have been used extensively in mathematics and theoretical physics, in particular in the study of solvable statistical models and quantum field theory. The theory of Yangian algebras has ample applications in
statistical mechanics, representation theory and algebraic combinatorics, for example
many interesting combinatorial properties of the general linear groups have been generalized to
the Yangian algebra of type $A$.
For a beautiful comprehensive introduction, see Molev's monograph \cite{M} (also \cite{MNO}).

There are three important realizations of the Yangians. The first one was
given by Drinfeld using generators and relations similar to Serre relations
in his fundamental paper \cite{D1} in 1985. Drinfeld pointed out that the Yangian algebra is the unique homogeneous quantization
of the half-loop algebra $\mathfrak{g}[u]$ associated to a simple Lie algebra $\mathfrak{g}$. Further along
 this development, Drinfeld \cite{D2} gave the second realization called Drinfeld realization for both quantum affine algebras and Yangians,
which can be viewed as the quantum analogue of the loop realization of the affine Lie algebras. Using this new
realization of Yangians,
Drinfeld developed and classified finite dimensional irreducible representations of Yangians.
The third realization is the generalization of the Faddeev-Reshetikhin-Takhatajan realization
which has a longer history originated from the quantum inverse
scattering method \cite{TF}. Under this realization the comultiplication formulas
have a particular simple form for both quantum affine algebras and Yangians.

Drinfeld stated that the FRT realization and Drinfeld's new realization of Yangian and quantum affine algebras are isomorphic. In 1994, Ding and Frenkel \cite{DF}
proved the isomorphism between the FRT realization and Drinfeld's new realization for quantum affine algebras of
type $A$. Later in 2005
Brundan and Kleshchev \cite{BK} proved the isomorphism between the two realizations of Yangians in type $A$.
For BCD type Yangians, Arnaudon, Molev and Ragoucy \cite{AMR} studied
the FRT realization of the Yangian and classified finite dimensional representations and gave
PBW theorem for the Yangian.  In \cite{AACFR}, Arnaudon et al proved that the FRT realization of BCD type Yangians
are indeed a homogeneous quantization of $\mathfrak{g}[u]$ corresponding to its canonical bialgebra structure.
Still less is known for the relations between Drinfeld's new realization and the FRT realization of BCD type Yangians.
Several known constructions or isomorphisms of Lie algebras are usually exclusively for Lie algebras.
For example, the Yangians associated to the orthogonal and symplectic Lie algebras are not known to be
subalgebras of the Yangian algebra associated to the general linear algebras. In fact this is why Olshanskii's twisted Yangians are introduced by extending the imbedding \cite{O1}.

In the lower rank cases, Arnaudon, Molev and Ragoucy proved the isomorphism between the FRT realization of $Y(\mathfrak{so}_3)$ and $Y({\mathfrak{sl}_2})$
in \cite{AMR} by using the fusion procedure for R-matrix.
In this paper we will prove the
isomorphism between the two realizations of $Y(\mathfrak{so}_3)$ by using the Gauss decomposition similar to
Ding-Frenkel's method \cite{DF}. It would be interesting to generalize this construction to
higher rank cases.

The paper is organized as follows. In the second section, we will give a brief introduction of the FRT or RTT realization of $Y(\mathfrak{so}_3)$. In the third section, we will study the Gauss decomposition of $Y(\mathfrak{so}_3)$ and give the relations between the Gauss generators. In the last section, we recall the Drinfeld's new realization and prove the main theorem.

\section{RTT realization of Yangian algebras corresponding to orthogonal Lie algebras}
In \cite{AMR} the RTT presentation of Yangian corresponding to orthogonal Lie algebras was studied.
Here we briefly recall the RTT realization of the Yangian with respect to orthogonal Lie algebras, especially the case of $\mathfrak{so}_3$.
Before introducing the RTT presentation, we list the following preliminaries and notations.
\vskip.1in

For $N=2n$ or $2n+1$ we enumerate the rows and columns of $N\times N$ matrices by the indices
$-n,...,-1,1,...,n$ and $-n,...-1,0,1,...n$ respectively
Let $t: End\mathbb{C}^N\rightarrow End\mathbb{C}^N$ be the transposition given by
\begin{equation*}
(e_{ij})^t=e_{-j,-i}.
\end{equation*}
\vskip.1in
Set $\kappa=N/2-1$, and consider the R-matrix
\begin{equation}
R(u)=1-\frac{P}{u}+\frac{Q}{u-\kappa},
\end{equation}
where $P=\sum_{i,j=-n}^ne_{ij}\otimes e_{ji}$, and $Q=P^{t}=\sum_{i,j=-n}^{n}e_{ij}\otimes e_{-i,-j}$.
It is easy to check that $R(u)$ is a rational solution of Yang-Baxter equation:
\begin{equation*}
R_{12}(u-v)R_{13}(u)R_{23}(v)=R_{23}(v)R_{13}(u)R_{12}(u-v).
\end{equation*}
\vskip.2in
We define an algebra $Y_R(\mathfrak{so}_N)$ by using the R-matrix $R(u)$ as follows.
\begin{definition}
The algebra $Y_R(\mathfrak{so}_N)$ is a unital associative algebra generated by $t^{(r)}_{ij}$, $r\in \mathbb{Z}_{+}$, $i,j\in\{-n,-n+1,...,n-1,n\}$,
the defining relations can be written as the following RTT form by the matrix of generators $T(u)$:
\begin{equation}\label{RTT}
R(u-v)T_1(u)T_2(v)=T_2(v)T_1(u)R(u-v),
\end{equation}
\begin{equation}
T(u)T^t(u+\kappa)=T^t(u+\kappa)T(u)=1,
\end{equation}\label{unitary relation}
where $T(u)=\sum_{i,j=-n}^n t_{ij}(u)\otimes e_{ij}$, $t_{ij}(u)=\sum_{r=0}^\infty t_{ij}^{(r)}u^{-r}$, $t^{(0)}_{ij}=\delta_{ij}$.
\end{definition}

Especially, for the $\mathfrak{so}_3$ case the algebra $Y_R(so_3)$ is generated by $t^{(r)}_{ij}$, $i,j\in\{-1,0,1\}$
and subject to the RTT relations.
\begin{proposition}
For the case $Y(\mathfrak{so}_3)$, the RTT generating relation (\ref{RTT}) can be written equivalently in terms of generating series as following:
\begin{equation}\label{generating relation1}
\begin{aligned}
&[t_{ij}(u),t_{kl}(v)]=\frac{1}{u-v}(t_{kj}(u)t_{il}(v)-t_{kj}(v)t_{il}(u))\\
&-\frac{1}{u-v-\frac{1}{2}}(\delta_{k,-i}\sum^{1}_{p=-1}t_{pj}(u)t_{-pl}(v)-\delta_{l,-j}\sum^{1}_{p=-1}t_{k,-p}(v)t_{ip}(u)).
\end{aligned}
\end{equation}
\end{proposition}
\vskip.2in
The RTT defining relation (\ref{RTT}) can be rewritten equivalently as follows:
\begin{equation}\label{RTT2}
T^{-1}_2(v)R(u-v)T_1(u)=T_1(u)R(u-v)T^{-1}_2(v).
\end{equation}
\vskip.2in
Denote by $t'_{ij}(u)$ the ij-th element of $T^{-1}(u)$.

\begin{proposition}
The defining relation (\ref{RTT}) is equivalent to
\begin{equation}\label{generating relation2}
\begin{aligned}
&[t_{pq}(u),t'_{rs}(v)]=\frac{1}{u-v-\frac{1}{2}}(t'_{r,-p}(v)t_{-s,q}(u)-t_{p,-r}(u)t'_{-q,s}(v))\\
&+\frac{1}{u-v}(\delta_{qr}\sum^{1}_{i=-1}t_{pi}(u)t'_{is}(v)-\delta_{ps}\sum^{1}_{i=-1}t'_{ri}(v)t_{iq}(u)).
\end{aligned}
\end{equation}
\end{proposition}

\section{Gauss decomposition of $Y_R(\mathfrak{so}_3)$}
In this section, we will study the Gauss decomposition of $T(u)$ and the commutation relations between the
``Gauss generators''.
\begin{theorem}\label{Gauss decth}
In $Y_R(\mathfrak{so}_3)$ the matrix $T(u)$ has the following unique decomposition:
\begin{align}\label{Gauss decomposition}
T(u)=\begin{pmatrix}
       1 & 0 & 0 \\
       f_{0,-1}(u) & 1 & 0 \\
       f_{1,-1}(u) & f_{1,0}(u) & 1 \\
     \end{pmatrix}
     \begin{pmatrix}
       k_{-1}(u) &  &  \\
        & k_{0}(u) &  \\
        &  & k_1(u)\\
     \end{pmatrix}
     \begin{pmatrix}
       1 & e_{-1,0}(u) & e_{-1,1}(u) \\
       0 & 1 & e_{01}(u) \\
       0 & 0 & 1 \\
     \end{pmatrix},
\end{align}
where the entries are defined by the matrix decomposition and for $i<j$
\begin{align*}
&e_{ij}(u)=\sum^{\infty}_{r=1}e^{(r)}_{ij}u^{-r}\in Y_R(\mathfrak{so}_3)[[u^{-1}]],\\
&f_{ji}(u)=\sum^{\infty}_{r=1}f^{(r)}_{ji}u^{-r}\in Y_R(\mathfrak{so}_3)[[u^{-1}]],\\
&k_{i}(u)=1+\sum^{\infty}_{r=1}k^{(r)}_{i}u^{-r}\in Y_R(\mathfrak{so}_3)[[u^{-1}]].
\end{align*}
The elements $e_{ij}(u)$, $f_{ji}(u)$, $k_i(u)$ are called the Gauss generators of
$Y_R(\mathfrak{so}_3)$.
\end{theorem}
\begin{proof} The decomposition is obtained formally by the matrix factorization. The matrix entries
in the decomposition can be computed as in the usual matrix algebra, which also provides the definition
for each of the Gauss generators iteratively.
Since $k_{-1}(u)$ is invertible, it is easy to prove that the Gauss-decomposition is unique.

\end{proof}

The commutation relations between the Gauss generators are directly computed from the matrix equation.
\vskip.2in

It follows from theorem \ref{Gauss decth} that
\begin{equation}
T(u)=\begin{pmatrix}
       k_{-1}(u) & k_{-1}(u)e_{-1,0}(u) & k_{-1}(u)e_{-1,1}(u) \\
       f_{0,-1}(u)k_{-1}(u) & k_{0}(u)+f_{0,-1}(u)k_{-1}(u)e_{-1,0}(u) & * \\
       f_{1,-1}(u)k_{-1}(u) & * & * \\
     \end{pmatrix},
\end{equation}
since $k_i(u)'s$ are invertible, we also obtain that
\begin{equation}
T^{-1}(u)=\begin{pmatrix}\label{eq3.3}
       * & * & * \\
       * & k_{0}^{-1}(u)+e_{01}(u)k_{1}^{-1}(u)f_{10}(u) & -e_{01}(u)k_{1}^{-1}(u) \\
       * & -k_{1}^{-1}(u)f_{10}(u) & k_{1}^{-1}(u) \\
     \end{pmatrix}.
\end{equation}
Moreover, from the defining relation (\ref{unitary relation}), we have $T^{t}(u+\frac{1}{2})=T^{-1}(u)$
and
\begin{align}\label{eq3.4}
T^{t}(u+\frac{1}{2})
=\begin{pmatrix}
                             * & * & * \\
                              *& \triangle & k_{-1}(u+\frac{1}{2})e_{-1,0}(u+\frac{1}{2}) \\
                             * & f_{0,-1}(u+\frac{1}{2})k_{-1}(u+\frac{1}{2}) & k_{-1}(u+\frac{1}{2}) \\
                           \end{pmatrix}.
\end{align}
where $\triangle=k_{0}(u+\frac{1}{2})+f_{0,-1}(u+\frac{1}{2})k_{-1}(u+\frac{1}{2})e_{-1,0}(u+\frac{1}{2})$.

Comparing the matrices (\ref{eq3.3}) and (\ref{eq3.4}), we get the following equations
in $Y_R(\mathfrak{so}_3)$:
\begin{equation}\label{eq3.5}
k_{1}^{-1}(u)=k_{-1}(u+\frac{1}{2})
\end{equation}
\begin{equation}\label{eq3.6}
-e_{01}(u)k_{1}^{-1}(u)=k_{-1}(u+\frac{1}{2})e_{-1,0}(u+\frac{1}{2})
\end{equation}

\begin{equation}\label{eq3.7}
-k_{1}^{-1}(u)f_{10}(u)=f_{0,-1}(u+\frac{1}{2})k_{-1}(u+\frac{1}{2})
\end{equation}

\begin{equation}\label{eq3.8}
k_{0}^{-1}(u)+e_{01}(u)k_{1}^{-1}(u)f_{10}(u)=k_{0}(u+\frac{1}{2})+f_{0,-1}(u+\frac{1}{2})k_{-1}(u+\frac{1}{2})e_{-1,0}(u+\frac{1}{2})
\end{equation}

Let's return back to the defining relations (\ref{generating relation1}) to obtain the
commutation relations among $k_{-1}(u),k_0(u)$ and $e_{-1,0}(u),f_{0,-1}(u)$.
\begin{proposition}\label{Prop3.2}
In $Y_R(\mathfrak{so}_3)$, we have
\begin{equation}\label{eq3.9}
[k_{-1}(u),k_{-1}(v)]=0 ,[k_{-1}(u),k_0(v)]=0,
\end{equation}
\begin{equation}\label{eq3.10}
[k_{-1}(u),e_{-1,0}(v)]=\frac{k_{-1}(u)(e_{-1,0}(v)-e_{-1,0}(u))}{u-v},
\end{equation}

\begin{equation}\label{eq3.11}
[k_{-1}(u),f_{0,-1}(v)]=\frac{(f_{0,-1}(u)-f_{0,-1}(v))k_{-1}(u)}{u-v}.
\end{equation}
\end{proposition}

\begin{proof} We just show (\ref{eq3.10}) as the other relations are proved in the same way.
It follows from equation (\ref{generating relation1}) that
\begin{equation*}
[t_{-1,-1}(u),t_{-1,0}(v)]=\frac{1}{u-v}(t_{-1,-1}(u)t_{-1,0}(v)-t_{-1,-1}(v)t_{-1,0}(u))
\end{equation*}
i.e. $[k_{-1}(u),k_{-1}(v)e_{-1,0}(v)]=\frac{1}{u-v}(k_{-1}(u)k_{-1}(v)e_{-1,0}(v)
-k_{-1}(v)k_{-1}(u)e_{-1,0}(u))
$
then formula (\ref{eq3.10}) follows immediately. Similarly we get Eq. (\ref{eq3.11}).
\end{proof}

In the following proposition, we will give the relations between $e_{-1,0}$ and $f_{0,-1}$.
\begin{proposition}\label{Prop3.3}
\begin{equation}\label{eq3.12}
[e_{-1,0}(u),f_{0,-1}(v)]=\frac{k_{-1}^{-1}(u)k_{0}(u)-k_{-1}^{-1}(v)k_{0}(v)}{u-v}.
\end{equation}
\end{proposition}
\vskip.1in
\begin{proof} Since $t_{-1, 0}(u)=k_{-1}(u)e_{-1,0}(u)$ and $t_{-1, -1}(u)=k_{-1}(v)$,
relation (\ref{generating relation1}) implies that
\begin{align*}
&k_{-1}(u)e_{-1,0}(u)k_{-1}(v)-k_{-1}(v)k_{-1}(u)e_{-1,0}(u)=\\
&\frac{1}{u-v}(k_{-1}(u)e_{-1,0}(u)k_{-1}(v)-k_{-1}(v)e_{-1,0}(v)k_{-1}(u)).\\
\end{align*}

Thus, we have
\begin{equation}\label{eq3.13}
\begin{aligned}
-f_{0,-1}(v)k_{-1}(v)k_{-1}(u)e_{-1,0}(u)=
-\frac{u-v-1}{u-v}f_{0,-1}(v)k_{-1}(u)e_{-1,0}(u)k_{-1}(v)\\
-\frac{1}{u-v}f_{0,-1}(v)k_{-1}(v)e_{-1,0}(v)k_{-1}(u)
\end{aligned}
\end{equation}

Using the equation (\ref{eq3.11}), we obtain
\begin{align*}
f_{0,-1}(v)k_{-1}(u)=\frac{u-v}{u-v-1}k_{-1}(u)f_{0,-1}(v)-\frac{1}{u-v-1}f_{0,-1}(u)k_{-1}(u).
\end{align*}
Then from the equation (\ref{eq3.13}), we have
\begin{equation}\label{eq3.14}
\begin{aligned}
-f_{0,-1}(v)k_{-1}(v)k_{-1}(u)e_{-1,0}(u)=-k_{-1}(u)f_{0,-1}(v)e_{-1,0}(u)k_{-1}(v)\\
+\frac{1}{u-v}f_{0,-1}(u)k_{-1}(u)e_{-1,0}(u)k_{-1}(v)\\
-\frac{1}{u-v}f_{0,-1}(v)k_{-1}(v)e_{-1,0}(v)k_{-1}(u)
\end{aligned}
\end{equation}
Similarly using $t_{0,-1}(u)=k_{-1}(u)e_{-1,0}(u)$ and $t_{-1, 0}(u)=f_{0,-1}(u)k_{-1}(u)$
we get the following from equation (\ref{generating relation1}):
\begin{equation}\label{eq3.15}
\begin{aligned}
k_{-1}(u)e_{-1,0}(u)f_{0,-1}(v)k_{-1}(v)-f_{0,-1}(v)k_{-1}(v)k_{-1}(u)e_{-1,0}(u)=\\
\frac{1}{u-v}(k_0(u)k_{-1}(v)-k_{0}(v)k_{-1}(u))\\
+\frac{1}{u-v}f_{0,-1}(u)k_{-1}(u)e_{-1,0}(u)k_{-1}(v)\\
-\frac{1}{u-v}f_{0,-1}(v)k_{-1}(v)e_{-1,0}(v)k_{-1}(u).\\
\end{aligned}
\end{equation}

Plugging Eq. (\ref{eq3.14}) into Eq. (\ref{eq3.15}), we have
\begin{equation}\label{eq3.16}
\begin{aligned}
k_{-1}(u)e_{-1,0}(u)f_{0,-1}(v)k_{-1}(v)-k_{-1}(u)f_{0,-1}(v)e_{-1,0}(u)k_{-1}(v)=\\
\frac{1}{u-v}(k_0(u)k_{-1}(v)-k_{0}(v)k_{-1}(u)).
\end{aligned}
\end{equation}
Since the $k_{i}(u)'s$ are invertible, we get Eq. (\ref{eq3.12}).
\end{proof}

The following relations between $e_{-1,0}(u)$ and $e_{01}(u)$ ($f_{0,-1}(u)$ and $f_{10}(u)$) will be useful.

\begin{proposition}\label{Prop3.4} In the algebra $Y_R(\mathfrak{so}_3)$
\begin{equation}\label{eq3.17}
e_{01}(u)=-e_{-1,0}(u-\frac{1}{2})
\end{equation}
\begin{equation}\label{eq3.18}
f_{10}(u)=-f_{0,-1}(u-\frac{1}{2})
\end{equation}
\end{proposition}

\begin{proof} From Eq. (\ref{eq3.10}), we easily get
\begin{equation}\label{eq3.19}
k_{-1}(u+\frac{1}{2})e_{-1,0}(u+\frac{1}{2})=e_{-1,0}(u-\frac{1}{2})k_{-1}(u+\frac{1}{2}).
\end{equation}
Taking the account of (\ref{eq3.5}) and (\ref{eq3.6}), we have $e_{01}(u)=-e_{-1,0}(u-\frac{1}{2})$.
Eq. (\ref{eq3.18}) is proved similarly.
\end{proof}

Now we can give the relations between $k_0(u)$ and $k_{-1}(u)$.
\begin{proposition}\label{Prop3.5}
In $Y_R(\mathfrak{so}_3)$, we have
\begin{equation}\label{eq3.20}
k_0(u)=k_{-1}(u)k^{-1}_{-1}(u+\frac{1}{2})
\end{equation}
\end{proposition}

\begin{proof}
If follows from Eq. (\ref{eq3.8}) that
\begin{equation*}
k_{0}(u+\frac{1}{2})-k_{0}^{-1}(u)=e_{01}(u)k_{1}^{-1}(u)f_{10}(u)-f_{0,-1}(u+\frac{1}{2})k_{-1}(u+\frac{1}{2})e_{-1,0}(u+\frac{1}{2}).
\end{equation*}
Moreover, using the equations (\ref{eq3.7}), (\ref{eq3.17}) and (\ref{eq3.19}), we can get
\begin{equation*}
k_{0}(u+\frac{1}{2})-k^{-1}_{0}(u)=[e_{-1,0}(u-\frac{1}{2}),f_{0,-1}(u+\frac{1}{2})]k_{-1}(u+\frac{1}{2}).
\end{equation*}

Proposition \ref{Prop3.3} implies that
\begin{equation*}
[e_{-1,0}(u-\frac{1}{2}),f_{0,-1}(u+\frac{1}{2})]=-(k_{-1}^{-1}(u-\frac{1}{2})k_{0}(u-\frac{1}{2})-k_{-1}^{-1}(u+\frac{1}{2})k_{0}(u+\frac{1}{2})),
\end{equation*}
thus,
\begin{equation*}
k^{-1}_{0}(u)=k_{-1}^{-1}(u-\frac{1}{2})k_{0}(u-\frac{1}{2})k_{-1}(u+\frac{1}{2}),
\end{equation*}
which is equivalent to
\begin{equation*}
k_{0}(u-\frac{1}{2})k_{0}(u)=k_{-1}(u-\frac{1}{2})k^{-1}_{-1}(u)k_{-1}(u)k^{-1}_{-1}(u+\frac{1}{2}).
\end{equation*}
Since $k_i(u)'s$ are formal series, we get
\begin{equation*}
k_{0}(u)=k_{-1}(u)k^{-1}_{-1}(u+\frac{1}{2}).
\end{equation*}

\end{proof}

\begin{proposition}\label{prop3.6}
In $Y_R(\mathfrak{so}_3)$, we have
\begin{equation}\label{eq3.21}
[k^{-1}_{-1}(u)k_{0}(u),e_{-1,0}(v)]=\frac{1}{2}\frac{1}{u-v}\{k^{-1}_{-1}(u)k_{0}(u),e_{-1,0}(u)-e_{-1,0}(v)\}
\end{equation}
\begin{equation}\label{eq3.22}
[k^{-1}_{-1}(u)k_{0}(u),f_{0,-1}(v)]=-\frac{1}{2}\frac{1}{u-v}\{k^{-1}_{-1}(u)k_{0}(u),f_{0,-1}(u)-f_{0,-1}(v)\}
\end{equation}
\end{proposition}

\begin{proof} Since two formulas are treated similarly,
we only prove the formula (\ref{eq3.21}).
From proposition \ref{Prop3.5} we obtain $k^{-1}_{-1}(u)k_{0}(u)=k^{-1}_{-1}(u+\frac{1}{2})$. Therefore
\begin{equation*}
[k^{-1}_{-1}(u)k_{0}(u),e_{-1,0}(v)]=[k^{-1}_{-1}(u+\frac{1}{2}),e_{-1,0}(v)].
\end{equation*}
Furthermore, from Eq. (\ref{eq3.10}), we get
\begin{equation*}
[k^{-1}_{-1}(u+\frac{1}{2}),e_{-1,0}(v)]=\frac{1}{u-v+\frac{1}{2}}(e_{-1,0}(u+\frac{1}{2})-e_{-1,0}(v))k^{-1}_{-1}(u+\frac{1}{2}).
\end{equation*}
Now the proposition will follow if the following is true.
\begin{equation*}
\begin{aligned}
\frac{1}{u-v+\frac{1}{2}}(e_{-1,0}(u+\frac{1}{2})-e_{-1,0}(v))k^{-1}_{-1}(u+\frac{1}{2})\\
=\frac{1}{2}\frac{1}{u-v}\{k^{-1}_{-1}(u+\frac{1}{2}),e_{-1,0}(u)-e_{-1,0}(v))\}
\end{aligned}
\end{equation*}
In fact, using Eq. (\ref{eq3.10}) one has
\begin{equation*}
k^{-1}_{-1}(u+\frac{1}{2})e_{-1,0}(u)=2e_{-1,0}(u+\frac{1}{2})k^{-1}_{-1}(u+\frac{1}{2})-e_{-1,0}(u)k^{-1}_{-1}(u+\frac{1}{2}),
\end{equation*}

\begin{equation*}
k^{-1}_{-1}(u+\frac{1}{2})e_{-1,0}(v)=\frac{1}{u-v+\frac{1}{2}}e_{-1,0}(u+\frac{1}{2})k^{-1}_{-1}(u+\frac{1}{2})+
\frac{u-v-\frac{1}{2}}{u-v+\frac{1}{2}}e_{-1,0}(v)k^{-1}_{-1}(u+\frac{1}{2}).
\end{equation*}
Hence, we can obtain the difference of the above two equations
\begin{equation*}
\begin{aligned}
&k^{-1}_{-1}(u+\frac{1}{2})(e_{-1,0}(u)-e_{-1,0}(v))=\\
&\frac{2(u-v)}{u-v+\frac{1}{2}}e_{-1,0}(u+\frac{1}{2})k^{-1}_{-1}(u+\frac{1}{2})
-e_{-1,0}(u)k^{-1}_{-1}(u+\frac{1}{2})\\
&-\frac{u-v-\frac{1}{2}}{u-v+\frac{1}{2}}e_{-1,0}(v)k^{-1}_{-1}(u+\frac{1}{2}).
\end{aligned}
\end{equation*}
Thus, we get
\begin{equation*}
\{k^{-1}_{-1}(u+\frac{1}{2}),e_{-1,0}(u)-e_{-1,0}(v)\}=\frac{2(u-v)}{u-v+\frac{1}{2}}(e_{-1,0}(u+\frac{1}{2})-e_{-1,0}(v))k^{-1}_{-1}(u+\frac{1}{2}),
\end{equation*}
which is just we need.
\end{proof}

\begin{proposition}\label{prop3.7}
In $Y_R(\mathfrak{so}_3)$, we have
\begin{equation}\label{eq3.23}
[e_{-1,0}(u),e_{-1,0}(v)]=\frac{1}{2}\frac{(e_{-1,0}(u)-e_{-1,0}(v))^2}{u-v},
\end{equation}

\begin{equation}\label{eq3.24}
[f_{0,-1}(u),f_{0,-1}(v)]=-\frac{1}{2}\frac{(f_{0,-1}(u)-f_{0,-1}(v))^2}{u-v}.
\end{equation}
\end{proposition}
Before proving Proposition \ref{prop3.7}, we first derive the following useful result.
\begin{lemma}
In $Y_R(\mathfrak{so}_3)$ one has
\begin{equation*}
3e_{-1,1}(u+\frac{1}{2})-e_{-1,1}(u)+3e_{-1,0}(u+\frac{1}{2})e_{-1,0}(u)-2e^{2}_{-1,0}(u)=0.
\end{equation*}

\end{lemma}
\begin{proof} First we recall the usual generators from Eq. (\ref{generating relation1}).
\begin{equation*}
\begin{aligned}
&[t_{-1,-1}(u),t_{-1,1}(u+\frac{1}{2})]=-2t_{-1,-1}(u)t_{-1,1}(u+\frac{1}{2})+2t_{-1,-1}(u+\frac{1}{2})t_{-1,1}(u)\\
&-t_{-1,1}(u+\frac{1}{2})t_{-1,-1}(u)-t_{-1,0}(u+\frac{1}{2})t_{-1,0}(u)-t_{-1,-1}(u+\frac{1}{2})t_{-1,1}(u),
\end{aligned}
\end{equation*}
which is actually equivalent to the following equation:
\begin{equation*}
3t_{-1,-1}(u)t_{-1,1}(u+\frac{1}{2})-t_{-1,-1}(u+\frac{1}{2})t_{-1,1}(u)+t_{-1,0}(u+\frac{1}{2})t_{-1,0}(u)=0.
\end{equation*}
Rewriting the above equation in terms of Gauss generators we get
\begin{equation}\label{eq3.25}
\begin{aligned}
&3k_{-1}(u)k_{-1}(u+\frac{1}{2})e_{-1,1}(u+\frac{1}{2})-k_{-1}(u+\frac{1}{2})k_{-1}(u)e_{-1,1}(u)\\
&+k_{-1}(u+\frac{1}{2})e_{-1,0}(u+\frac{1}{2})k_{-1}(u)e_{-1,0}(u)=0.
\end{aligned}
\end{equation}
Then it follows from Eq. (\ref{eq3.10}) that
\begin{equation*}
e_{-1,0}(u+\frac{1}{2})k_{-1}(u)=3k_{-1}(u)e_{-1,0}(u+\frac{1}{2})-2k_{-1}(u)e_{-1,0}(u).
\end{equation*}
Thus Eq. (\ref{eq3.25}) is equivalent to
\begin{equation*}
k_{-1}(u)k_{-1}(u+\frac{1}{2})(3e_{-1,1}(u+\frac{1}{2})-e_{-1,1}(u)+3e_{-1,0}(u+\frac{1}{2})e_{-1,0}(u)-2e^{2}_{-1,0}(u))=0.
\end{equation*}
This implies the result as $k_{i}(u)$'s are invertible.

\end{proof}

\begin{lemma}
In $Y_R(\mathfrak{so}_3)$, we have
\begin{equation*}
e_{-1,1}(u)=[e^{(1)}_{-1,0},e_{-1,0}(u)]-e^{2}_{-1,0}(u)
\end{equation*}
\end{lemma}

\begin{proof}
It follows from Eq. (\ref{generating relation1}) that
\begin{equation*}
[t_{-1,0}(u),t_{01}(v)]=\frac{1}{u-v}(t_{00}(u)t_{-1,1}(v)-t_{00}(v)t_{-1,1}(u)).
\end{equation*}
Taking $v\rightarrow \infty$ we obtain
\begin{equation*}
[t_{-1,0}(u),t^{(1)}_{01}]=t_{-1,1}(u).
\end{equation*}
Using the Gauss-decomposition of $T(v)$, we know that
\begin{align*}
t_{01}(v)=k_{0}(v)e_{01}(v)+f_{0,-1}(v)k_{-1}(v)e_{-1,0}(v).
\end{align*}
So $t^{(1)}_{01}=e^{(1)}_{01}$. Moreover proposition \ref{Prop3.4}
implies that
$e^{(1)}_{01}=-e^{(1)}_{-1,0}.$
Subsequently
\begin{equation*}
[k_{-1}(u)e_{-1,0}(u),e^{(1)}_{-1,0}]=-k_{-1}(u)e_{-1,1}(u).
\end{equation*}
From Eq. (\ref{eq3.10}) one then gets that
\begin{equation*}
[k_{-1}(u),e^{(1)}_{-1,0}]=k_{-1}(u)e_{-1,0}(u).
\end{equation*}
Since $$[k_{-1}(u)e_{-1,0}(u),e^{(1)}_{-1,0}]=k_{-1}(u)[e_{-1,0}(u),e^{(1)}_{-1,0}]+[k_{-1}(u),e^{(1)}_{-1,0}]e_{-1,0}(u),$$
we have
\begin{equation*}
k_{-1}(u)[e_{-1,0}(u),e^{(1)}_{-1,0}]+k_{-1}(u)e^{2}_{-1,0}(u)=-k_{-1}(u)e_{-1,1}(u).
\end{equation*}
Dividing $k_{i}(u)'s$ we prove the lemma.
\end{proof}

\begin{lemma}
In $Y_R(\mathfrak{so}_3)$ we have
\begin{equation*}
[e^{(1)}_{-1,0},e_{-1,0}(u)]=e^{2}_{-1,0}(u)-e_{-1,0}(u+\frac{1}{2})e_{-1,0}(u)-e_{-1,1}(u+\frac{1}{2}).
\end{equation*}
\end{lemma}

\begin{proof}
From Eq. (\ref{generating relation2}) we compute that
\begin{equation*}
\begin{aligned}
&[t_{-1,0}(u),t'_{01}(v)]=\frac{1}{u-v-\frac{1}{2}}(t'_{01}(v)t_{-1,0}(u)-t_{-1,0}(u)t'_{01}(v))\\
&+\frac{1}{u-v}(t_{-1,-1}(u)t'_{-1,1}(v)+t_{-1,0}(u)t'_{01}(v)+t_{-1,1}(u)t'_{11}(v)).
\end{aligned}
\end{equation*}
which is equivalent to
\begin{equation*}
(u-v)\frac{u-v+\frac{1}{2}}{u-v-\frac{1}{2}}[t_{-1,0}(u),t'_{01}(v)]=
t_{-1,-1}(u)t'_{-1,1}(v)+t_{-1,0}(u)t'_{01}(v)+t_{-1,1}(u)t'_{11}(v).
\end{equation*}
Taking $u\rightarrow \infty$ and invoking Gauss decomposition of $T(u)$ and $T'(u)$ we obtain that
\begin{equation}\label{eq3.26}
-[e^{(1)}_{-1,0},e_{01}(v)k^{-1}_{1}(v)]=e_{-1,0}(v)e_{01}(v)k^{-1}_{1}(v)-e_{-1,1}(v)k^{-1}_{1}(v).
\end{equation}

Note that from Eq. (\ref{eq3.10}) we have
\begin{align*}
[e^{(1)}_{-1,0},k^{-1}_{1}(v)]=e_{01}(v)k^{-1}_{1}(v),
\end{align*}
then we derive
\begin{equation}\label{eq3.27}
-[e^{(1)}_{-1,0},e_{01}(v)k^{-1}_{1}(v)]=-[e^{(1)}_{-1,0},e_{01}(v)]k^{-1}_{1}(v)-e^{2}_{01}(v)k^{-1}_{1}(v).
\end{equation}
Combining equations (\ref{eq3.26}) and (\ref{eq3.27}) we get
\begin{equation*}
-[e^{(1)}_{-1,0},e_{01}(v)]-e^{2}_{01}(v)=e_{-1,0}(v)e_{01}(v)-e_{-1,1}(v).
\end{equation*}
Finally using Eq. (\ref{eq3.17}) we obtain the following
equation
\begin{equation*}
[e^{(1)}_{-1,0},e_{-1,0}(v-\frac{1}{2})]-e^{2}_{-1,0}(v-\frac{1}{2})=-e_{-1,0}(v)e_{-1,0}(v-\frac{1}{2})-e_{-1,1}(v).
\end{equation*}
The lemma is obtained if we take $u=v-\frac{1}{2}$.
\end{proof}

Next we can prove an interesting relation between $e_{-1,1}(u)$ and $e_{-1,0}(u)$.
\begin{lemma}\label{lemma3.11}
In $Y_R(\mathfrak{so}_3)$, we have
\begin{equation*}
e_{-1,1}(u)=-\frac{1}{2}e^{2}_{-1,0}(u).
\end{equation*}
\end{lemma}
\begin{proof}
Using previous lemmas, we have
\begin{equation*}
e_{-1,1}(u)=-e_{-1,0}(u+\frac{1}{2})e_{-1,0}(u)-e_{-1,1}(u+\frac{1}{2}).
\end{equation*}
Moreover we also have
\begin{equation*}
-e_{-1,0}(u+\frac{1}{2})e_{-1,0}(u)-e_{-1,1}(u+\frac{1}{2})=-\frac{1}{3}(e_{-1,1}(u)+2e^{2}_{-1,0}(u)).
\end{equation*}
The lemma then follows immediately.

\end{proof}

Similarly, we can get the relation between $f_{1,-1}(u)$ and $f_{1,0}(u)$.
\begin{lemma}\label{lemma3.12}
In $Y_R(\mathfrak{so}_3)$, we have
\begin{equation*}
f_{1,-1}(u)=-\frac{1}{2}f^{2}_{1,0}(u).
\end{equation*}
\end{lemma}

Now we are ready to prove Proposition \ref{prop3.7}.
\begin{proof} It is enough to check Eq. (\ref{eq3.23}), as Eq. (\ref{eq3.24}) can be treated similarly.

Using the generating relations (\ref{generating relation1}) we get
\begin{equation*}
\begin{aligned}
&[t_{-1,0}(u),t_{-1,0}(v)]=\frac{1}{u-v}(t_{-1,0}(u)t_{-1,0}(v)-t_{-1,0}(v)t_{-1,0}(u))+\\
&\frac{1}{u-v-\frac{1}{2}}(t_{-1,1}(v)t_{-1,-1}(u)+t_{-1,0}(v)t_{-1,0}(u)+t_{-1,-1}(v)t_{-1,1}(u)),
\end{aligned}
\end{equation*}

\begin{equation*}
\begin{aligned}
&[t_{-1,-1}(u),t_{-1,1}(v)]=\frac{1}{u-v}(t_{-1,-1}(u)t_{-1,1}(v)-t_{-1,-1}(v)t_{-1,1}(u))+\\
&\frac{1}{u-v-\frac{1}{2}}(t_{-1,1}(v)t_{-1,-1}(u)+t_{-1,0}(v)t_{-1,0}(u)+t_{-1,-1}(v)t_{-1,1}(u)).
\end{aligned}
\end{equation*}
Comparing these two equations we obtain that
\begin{equation}\label{eq3.28}
\begin{aligned}
&[t_{-1,0}(u),t_{-1,0}(v)]=t_{-1,-1}(u)t_{-1,1}(v)
+\frac{1}{u-v-1}t_{-1,-1}(v)t_{-1,1}(u)\\
&-\frac{u-v}{u-v-1}t_{-1,1}(v)t_{-1,-1}(u).
\end{aligned}
\end{equation}
Using the defining relations (\ref{generating relation1}) again, we have
\begin{equation*}
[t_{-1,-1}(v),t_{-1,0}(u)]=\frac{1}{v-u}(t_{-1,-1}(v)t_{-1,0}(u)-t_{-1,-1}(u)t_{-1,0}(v)).
\end{equation*}
In terms of Gauss generators they are
\begin{equation*}
\begin{aligned}
k_{-1}(u)e_{-1,0}(u)k_{-1}(v)&=k_{-1}(u)k_{-1}(v)e_{-1,0}(u)+
&\frac{1}{u-v}k_{-1}(u)k_{-1}(v)(e_{-1,0}(u)-e_{-1,0}(v)).
\end{aligned}
\end{equation*}
Since $t_{-1,0}(u)t_{-1,0}(v)=k_{-1}(u)e_{-1,0}(u)k_{-1}(v)e_{-1,0}(v)$,
we can write
\begin{align*}
t_{-1,0}(u)t_{-1,0}(v)&=k_{-1}(u)k_{-1}(v)e_{-1,0}(u)e_{-1,0}(v)\\
&+\frac{1}{u-v}k_{-1}(u)k_{-1}(v)(e_{-1,0}(u)e_{-1,0}(v)-e^2_{-1,0}(v)),
\end{align*}
and
\begin{align*}
t_{-1,0}(v)t_{-1,0}(u)&=k_{-1}(v)k_{-1}(u)e_{-1,0}(v)e_{-1,0}(u)\\
&+\frac{1}{v-u}k_{-1}(v)k_{-1}(u)(e_{-1,0}(v)e_{-1,0}(u)-e^2_{-1,0}(u)).
\end{align*}
Then we get to the left-hand side of Eq. (\ref{eq3.28})
\begin{align*}
[t_{-1,0}(u),t_{-1,0}(v)]&=k_{-1}(u)k_{-1}(v)[e_{-1,0}(u),e_{-1,0}(v)]\\
&-\frac{1}{u-v}k_{-1}(u)k_{-1}(v)(e_{-1,0}(u)
-e_{-1,0}(v))^2.
\end{align*}

Now we consider the right-hand side of Eq. (\ref{eq3.28}).
Using the Gauss-decomposition and Lemma \ref{lemma3.11} we have
\begin{equation*}
t_{-1,1}(v)t_{-1,-1}(u)=-\frac{1}{2}k_{-1}(v)e^{2}_{-1,0}(v)k_{-1}(u).
\end{equation*}
Using Eq. (\ref{eq3.10}) we finally get
\begin{equation*}
\begin{aligned}
e^{2}_{-1,0}(v)k_{-1}(u)&=\frac{(u-v-1)^2}{(u-v)^2}k_{-1}(u)e^{2}_{-1,0}(v)+\frac{u-v-1}{(u-v)^2}k_{-1}(u)e_{-1,0}(u)e_{-1,0}(v)\\
&+\frac{u-v-1}{(u-v)^2}k_{-1}(u)e_{-1,0}(v)e_{-1,0}(u)+\frac{1}{(u-v)^2}k_{-1}(u)e^{2}_{-1,0}(u).\\
\end{aligned}
\end{equation*}
Then using Gauss decomposition of $T(u)$ and Lemma (\ref{lemma3.11}) again, we obtain the right hand side of Eq. (\ref{eq3.28})
equals to
\begin{equation*}
-\frac{1}{2}\frac{1}{u-v}k_{-1}(u)k_{-1}(v)(e_{-1,0}(u)-e_{-1,0}(v))^2.
\end{equation*}
Taking the consideration of the LHS we prove the proposition. 
\end{proof}

\section{The isomorphism between RTT realization and Drinfeld's realization}
We recall the Drinfeld's realization of the Yangian associated to $\mathfrak{so}_3$,
denoted by $Y(\mathfrak{so}_3)$ in the following.
\begin{definition}
$Y(\mathfrak{so}_3)$ is the associative algebra generated by infinite generators
$h_l$, $x^{\pm}_l$, (l=0,1,2,...),
subject to the following defining relations:
\begin{equation*}
[h_k,h_l]=0, [x^+_k,x^-_l]=h_{k+l}, [h_0,x^{\pm}_l]=\pm x^{\pm}_{l},
\end{equation*}

\begin{equation*}
[h_{k+1},x^{\pm}_l]-[h_k,x^{\pm}_{l+1}]=\pm \frac{1}{2}\{h_k,x^{\pm}_{l}\},
\end{equation*}

\begin{equation*}
[x^{\pm}_{k+1},x^{\pm}_l]-[x^{\pm}_k,x^{\pm}_{l+1}]=\pm \frac{1}{2}\{x^{\pm}_k,x^{\pm}_{l}\}.
\end{equation*}
\end{definition}

Let $X^{\pm}(u)=\sum_{k=0}^{\infty}x^{\pm}_ku^{-k-1}$, $H(u)=1+\sum_{k=0}^{\infty}h_ku^{-k-1}$
be the generating series, then we have the following proposition.
\begin{proposition}\label{prop4.2}
The defining relations of $Y(\mathfrak{so}_3)$ are equivalent to the following form in terms of generating series:

\begin{equation}
[H(u),H(v)]=0,
\end{equation}

\begin{equation}\label{eq4.2}
[X^{+}(u),X^{-}(v)]=-\frac{H(u)-H(v)}{u-v},
\end{equation}

\begin{equation}
[H(u),X^{+}(v)]=-\frac{1}{2}\frac{\{H(u),(X^{+}(u)-X^{+}(v))\}}{u-v},
\end{equation}

\begin{equation}
[H(u),X^{-}(v)]=\frac{1}{2}\frac{\{H(u),(X^{-}(u)-X^{-}(v))\}}{u-v},
\end{equation}

\begin{equation}
[X^{+}(u),X^{+}(v)]=-\frac{1}{2}\frac{(X^{+}(u)-X^{+}(v))^2}{u-v},
\end{equation}

\begin{equation}
[X^{-}(u),X^{-}(v)]=\frac{1}{2}\frac{(X^{-}_{i}(u)-X^{-}_{i}(v))^2}{u-v}.
\end{equation}
\end{proposition}
\begin{remark}
The proof of proposition \ref{prop4.2} can be derived by comparing the coefficients of the two sides of equations.
\end{remark}
Now we can give the main result in the following theorem.
\begin{theorem}
The map $\Phi :Y(\mathfrak{so}_3)\rightarrow Y_R(\mathfrak{so}_3)$ given by
\begin{align}
X^{-}(u)\mapsto e_{-1,0}(u), X^{+}(u)\mapsto f_{0,-1}(u), H(u)\mapsto k_{-1}^{-1}(u)k_{0}(u)
\end{align}
is an isomorphism.
\end{theorem}

\begin{proof}
It is easy to verify that $\Phi$ is a homomorphism from Propositions \ref{Prop3.2}, \ref{Prop3.3}, \ref{prop3.6}, and \ref{prop3.7}.
Furthermore, the surjectivity is obtained by using the Gauss decomposition of $T(u)$ and Eq. (\ref{eq3.5}), Propositions
\ref{Prop3.4}, \ref{Prop3.5}, \ref{lemma3.11} and Lemma \ref{lemma3.12}. Since the Gauss generators $k_{-1}(u)$, $e_{-1,0}(u)$,
$f_{0,-1}(u)$ can generate the algebra $Y_R(\mathfrak{so}_3)$, the inverse map of $\Phi$ can be given by
\begin{align}
e_{-1,0}(u)\mapsto X^{-}(u) , f_{0,-1}(u)\mapsto X^{+}(u), k_{-1}(u)\mapsto H^{-1}(u-\frac{1}{2}).
\end{align}

\end{proof}

\medskip

\centerline{\bf Acknowledgments}
NJ gratefully acknowledges the partial support of Max-Planck Institut f\"ur Mathematik in Bonn, Simons Foundation grant 198129, and NSFC grant 11271138 during this work.

\bibliographystyle{amsalpha}

\end{document}